\documentclass[12pt,leqno,a4paper]{amsart}
\usepackage{amssymb,enumerate}
\newcommand{\CC}{{\mathbb{C}}}
\newcommand{\FF}{{\mathbb{F}}}
\newcommand{\ZZ}{{\mathbb{Z}}}
\newcommand{\QQ}{{\mathbb{Q}}}

\newcommand{\fA}{{\mathfrak{A}}}
\newcommand{\fS}{{\mathfrak{S}}}

\newcommand{\bC}{{\mathbf{C}}}
\newcommand{\bE}{{\mathbf{E}}}
\newcommand{\bF}{{\mathbf{F}}}
\newcommand{\bG}{{\mathbf{G}}}
\newcommand{\bH}{{\mathbf{H}}}
\newcommand{\bL}{{\mathbf{L}}}
\newcommand{\bO}{{\mathbf{O}}}
\newcommand{\bR}{{\mathbf{R}}}
\newcommand{\bT}{{\mathbf{T}}}
\newcommand{\bZ}{{\mathbf{Z}}}

\newcommand{\cE}{{\mathcal{E}}}
\newcommand{\cO}{{\mathcal{O}}}
\newcommand{\cS}{{\mathcal{S}}}
\newcommand{\cX}{{\mathcal{X}}}
\newcommand{\cY}{{\mathcal{Y}}}

\newcommand{\Aut}{{\operatorname{Aut}}}

\newcommand{\Diag}{{\operatorname{Diag}}}
\newcommand{\Gal}{{\operatorname{Gal}}}
\newcommand{\Inn}{{\operatorname{Inn}}}
\newcommand{\IBr}{{\operatorname{IBr}}}
\newcommand{\Irr}{{\operatorname{Irr}}}
\newcommand{\Mat}{{\operatorname{Mat}}}
\newcommand{\Out}{{\operatorname{Out}}}
\newcommand{\Syl}{{\operatorname{Syl}}}

\newcommand{\PSL}{{\operatorname{L}}}
\newcommand{\SL}{{\operatorname{SL}}}
\newcommand{\SU}{{\operatorname{SU}}}
\newcommand{\PSU}{{\operatorname{U}}}

\newcommand{\tw}[1]{{}^#1\!}
\newcommand{\tchi}{\tilde\chi}
\newcommand{\tla}{\tilde\lambda}
\newcommand{\RLG}{{R_\bL^\bG}}
\newcommand{\RTG}{{R_\bT^\bG}}
\newcommand{\w}{\widetilde}

\let\vhi=\varphi
\let\la=\lambda

\def\nor{\triangleleft\,}
\def\irr#1{{\rm Irr}(#1)}
\def\cent#1#2{{\bf C}_{#1}(#2)}
\def\norm#1#2{{\bf N}_{#1}(#2)}

\newtheorem{thm}{Theorem}[section]
\newtheorem{lem}[thm]{Lemma}

\newtheorem{cor}[thm]{Corollary}
\newtheorem{prop}[thm]{Proposition}

\newtheorem*{thmA}{Main Theorem}

\theoremstyle{definition}

\theoremstyle{remark}
\newtheorem{rem}[thm]{Remark}

\begin{document}

\title[Blocks with One Modular Character]{On Blocks with One Modular Character}

\date{\today}

\author{Gunter Malle}
\address{FB Mathematik, TU Kaiserslautern, Postfach 3049,
         67653 Kaisers\-lautern, Germany.}
\email{malle@mathematik.uni-kl.de}
\author{Gabriel Navarro}
\address{Departament d'\`Algebra, Universitat de Val\`encia, 46100 Burjassot,
         Val\`encia, Spain.}
\email{gabriel@uv.es}
\author{Britta Sp\"ath}
\address{FB Mathematik, TU Kaiserslautern, Postfach 3049,
         67653 Kaisers\-lautern, Germany.}
\email{spaeth@mathematik.uni-kl.de}

\begin{abstract}
Suppose that $B$ is a Brauer $p$-block of a finite group $G$ with a unique
modular character $\vhi$. We prove that $\vhi$ is liftable to an ordinary
character of $G$ (which moreover is $p$-rational for odd $p$). This confirms
the basic set conjecture for these blocks.
\end{abstract}

\thanks{The first and third author gratefully acknowledge financial support
 by ERC Advanced Grant 291512. The second author is partially supported by the
 Spanish Ministerio de Educaci\'on y Ciencia Proyectos MTM2013-40464-P and
 Prometeo II/Generalitat Valenciana.}



\maketitle


\section{Introduction}   \label{sec:intro}

Let $G$ be a finite group and let $p$ be a prime. While the Brauer $p$-blocks
of $G$ consisting of one ordinary irreducible character were understood by
Richard Brauer himself a long time ago, the blocks containing just one modular
character are still a mystery. The study of the Brou\'e--Puig nilpotent blocks
\cite{BP} (a canonical example of blocks with one modular character and, in
some sense, the easiest) is already quite complicated.
But not every block with one modular character is nilpotent (except possibly in
quasi-simple groups.)

\medskip

From the point of view of algebras, blocks with one modular character are the
most natural. If $F$ is an algebraically closed field of characteristic $p$,
and $B$ is a $p$-block of $G$ (that is, an indecomposable two-sided ideal of
the group algebra $FG$), then the algebra $B$ has a unique irreducible
$F$-representation if and only if $B/{\bf J}(B)$ is a matrix algebra,
where ${\bf J}(B)$ is the Jacobson radical of $B$. We see that from several
perspectives, the blocks with one modular character constitute a fundamental
object that deserves study.
\medskip

If $B$ has a unique irreducible modular character, say $\vhi$, then the
decomposition numbers of the algebra $B$ are simply the numbers
$\chi(1)/\vhi(1)$, where $\chi$ runs over the irreducible complex characters
in $B$. No general interpretation of these numbers is known. (If $B$ is
nilpotent, for instance,  these are the irreducible character degrees of any
defect group of $B$, see \cite[Thm. 1.2(2)]{BP}.) Our main result is that
one of these numbers should always be 1. In other words, $\vhi$ should always
lift to an ordinary (complex) character of $G$.

\begin{thmA}
 Let $G$ be a finite group and let $p$ be a prime. Assume that $B$ is a
 $p$-block of $G$ and  $\IBr(B)=\{\vhi\}$. Then there exists $\chi\in\Irr(B)$
 such that $\chi(1)=\vhi(1)$.
 If $p\neq 2$, $\chi$ can be chosen to be $p$-rational.
\end{thmA}

Our Main Theorem was known for nilpotent blocks (by the work of Brou\'e--Puig
\cite{BP}). It is also known for $p$-solvable groups by the celebrated
Fong--Swan Theorem (while the $p$-rationality part follows from a classical
theorem of M. Isaacs \cite{Isa1}). Our proof of the general case uses the
Classification of Finite Simple Groups, and the main
obstacle is to understand the $p$-blocks $B$ of quasi-simple groups $G$
such that $\IBr(B)$ is contained in a single $\Aut(G)$-orbit. We believe that
their classification will turn out to be fundamental when further studying
blocks with one modular character.

This work started from an observation by R. Kessar and M. Linckelmann.
Recall that the basic set conjecture asserts that for every block $B$ of a
finite group $G$, there exists a subset $\mathcal B$ of $\irr G$ such that the
set $\mathcal B^0$ of their restrictions to $p$-regular elements is a
$\ZZ$-basis of the ring $\ZZ[\IBr(B)]$ of generalized Brauer characters of $B$.
(See for instance \cite{GH}.) In the case where $\IBr(B)=\{\varphi\}$, as
pointed out by Kessar and Linckelmann, the basic set conjecture is equivalent
to proving that $\varphi$ is liftable, which is our Main Theorem.

This liftability of $\varphi$ in our main result is, once again, a shadow of
some deeper, structural conjecture that R. Kessar and M. Linckelmann have
proposed: every block with one modular irreducible character is Morita
equivalent over the ring of $p$-adics to a block of a $p$-solvable group.
(This is consistent with their results in \cite{K12} and \cite{KL10}, and it
would imply the liftability statement in our Main Theorem, using the
Fong--Swan Theorem and the fact that blocks which are Morita equivalent over
the ring of $p$-adics have the same decomposition matrix, see
\cite[Sect.~2.2, Ex.~3]{B}.) This deeper conjecture, however, seems out of
reach with the present methods.

Finally, for $p$ odd, we show that in the situation of the Main Theorem one
can also choose the lift $\chi$ to be $p$-rational. Nevertheless we can't
answer the question of G.~R.~Robinson whether this $p$-rational lift is
unique, although we believe that this is the case. None of this seems to be
implied by the general Kessar--Linckelmann conjecture.

\section{A result on simple groups}   \label{sec:notLie}

In this section and the following we start our investigation of blocks of
finite quasi-simple groups all of whose irreducible Brauer characters lie in
a single orbit under the automorphism group. In Section~\ref{sec:main proof}
it will turn out that exactly those blocks are relevant for our inductive
proof.
One can see the importance of those blocks already in the following situation:
assume that in the situation of the Main Theorem there exists a quasi-simple
normal subgroup $X$ of $G$. Then the blocks of $X$ covered by a block of $G$
with one Brauer character have the above mentioned property.

Let $p$ be a prime. Our notation for blocks and characters of finite groups
follows \cite{N}. Hence we have fixed a maximal ideal $M$ of the ring of
algebraic integers containing $p$ with respect to which Brauer characters
of finite groups are constructed.

\begin{thm}   \label{thm:quasisimples}
 Suppose that $G$ is a quasi-simple group, $p$ is a prime with
 $p\nmid |\bZ(G)|$, and $B$ is a $p$-block of $G$. Let $A=\Aut(G)$.
 If $\IBr(B)$ consists of $A$-conjugates of $\vhi$, then there exists some
 $\chi \in \Irr(G)$ such that
 \begin{itemize}
  \item[\rm(1)] $\chi^0=\vhi$;
  \item[\rm(2)] $\chi$ is $p$-rational if $p\ne2$;
  \item[\rm(3)] the stabilisers $I_A(\chi)=I_A(\vhi)=:I$ coincide; and
  \item[\rm(4)] for $P \in\Syl_p(I)$, the character $\chi$ extends to some
   $\tilde \chi\in \Irr(G\rtimes P)$ with
   $p\nmid |\cent{G\rtimes P} G: \ker (\tilde \chi|_{\cent{G\rtimes P} G})|$.
 \end{itemize}
\end{thm}

The proof will be given in this section and the next, based on the
classification of finite simple groups. Assertions~(3) and~(4) will be crucial
in our inductive approach to arbitrary groups. Note that assertion~(2) implies
the statement in~(4) in case of odd $p$ by \cite[Thm.~(6.30)]{Isa}.

\subsection{General Observations}
We start by collecting some general observations, relating to our
Theorem~\ref{thm:quasisimples}, as well as about nilpotent blocks and their
characters.

\begin{lem}   \label{extendingclean}
 Suppose that $E \nor G$. Let $P$ be a $p$-subgroup of $G$ such that
 $P \cap E=Q \in \Syl_p(E)$, and assume that $\cent PE=1$.
 Let $\theta \in \irr E$ be $P$-invariant. Then $\theta$ extends to $EP$ if
 and only if $\theta$ has an extension $\tilde\theta$ to $H:=E\rtimes P$ such
 that $p\nmid |\cent{H} E: \ker(\tilde \theta|_{\cent{H}E})|$.
\end{lem}

\begin{proof}
Since $E$ is normal, $P$ acts on $E$ as automorphisms by conjugation. Let
$H=E \rtimes P$ be the corresponding semidirect product, and view $E$ and
$P$ as subgroups of $H$. Now, the map $f: H \rightarrow EP$ given by
$f(e,x)=ex$ is a surjective group homomorphism with kernel
$\Delta Q=\{(x,x^{-1})\,|\, x \in Q \}$, a normal
$p$-subgroup of $H$ isomorphic to $Q$, which intersects trivially with
both $E$ and $P$. In particular, $\Delta Q \subseteq \cent H E$.
By orders, $(\Delta Q)P$ is a Sylow $p$-subgroup of $H$ and
$(\Delta Q)P \cap \cent H E= (\Delta Q) \cent PE=\Delta Q$, so we see
that $\Delta Q$ is a Sylow $p$-subgroup of $\cent H E$. Using that $f$
induces an isomorphism $H/\Delta Q \rightarrow EP$, we easily see that
$\theta$ extends to $EP$ if and only if $\theta$ has an extension
$\tilde \theta$ to $H$ that contains $\Delta Q$ in its kernel. Since
$\Delta Q \in \Syl_p(\cent H E)$, the rest of the claim easily follows.
\end{proof}

Before considering the blocks of specific simple groups we
prove that Theorem \ref{thm:quasisimples} holds whenever $B$
is nilpotent or is covered by a nilpotent block (in a specific situation).

\begin{prop}   \label{prop:nilpotent}
 Let $G$ be a finite group, $N\lhd G$ and $B$ be a $G$-invariant nilpotent
 $p$-block of $N$. Then for the unique $\vhi\in\IBr(B)$ there exists some
 $G$-invariant $\psi\in\Irr(B)$ with $\psi^0=\vhi$ such that $\psi$ extends
 to $Q$, where $Q/N$ is a Sylow $p$-subgroup of $G/N$.
 Moreover if $p\neq 2$, $\psi$ can be chosen to be $p$-rational.
\end{prop}

\begin{proof}
By the characterisation of characters in nilpotent $p$-blocks when $p>2$ there
is a unique character $\psi\in\Irr(B)$ that is $p$-rational,
see \cite[Thm.~1.2]{BP}. This character has the claimed properties according to
\cite[Thm.~(6.30)]{I}.

When $p=2$ we use the following considerations: Let $D$ be a defect group
of $B$ and $B' $ the Brauer correspondent of $B$ in $\norm N D$.
Let $\psi'\in \Irr(B')$ be the unique character with $D\subseteq \ker(\psi')$.
Hence $\psi'$ is $H:=\norm G D$-invariant. Note that $\psi'$ extends to
$\norm Q D$ since $\psi'$ corresponds to a defect zero character of
$\norm N D/D$ that extends to $\norm Q D/D$ according to \cite[Ex. (3.10)]{N}.
According to the proof of \cite[Thm.~(8.28)]{N} there exists
a central extension $\hat H$ of $H$ by a $p'$-group $U$ with
$\norm N D\lhd \hat H$ such that $\psi'$ extends to some
$\tilde \psi' \in \Irr(\hat H)$. Using the cocycle defining $\hat H$ as
central extension of $\norm G D$ by $U$ we define $\hat G$ as a central
extension of $G$ such that $N\lhd \hat G$ and $\hat H\subseteq \hat G$.
Moreover let $\cO$ be the complete discrete valuation ring of a $p$-modular
system (that is big enough), $e\in \cO N$ and $e'\in \cO\norm N D$ be the
central-primitive idempotents of $B$ and $B'$.
According to \cite[1.20.3]{KP} the algebras $\cO\hat Ge$ and
$\cO\norm {\hat G} D e'$ are Morita equivalent.

Since the irreducible module affording $\psi'$ has multiplicity one in
$\cO\norm {\hat G} D e'$ and dimension $\psi'(1)^2$ there exists a character
$\tilde\psi\in\Irr(\hat G)$ afforded by
a submodule of $\cO\hat Ge$ that has multiplicity one in $\cO\hat Ge$.
Hence $\psi:=\tilde \psi_N$ is irreducible, $G$-invariant and $\psi^0=\vhi$.
It remains to prove that $\psi$ extends to $Q$, but by the construction
$\widehat G$ is a central extension of $G$ by a $p'$-group,
hence $Q$ is (isomorphic to) a subgroup of $\hat G$ and $\psi$ extends to $Q$.
\end{proof}

\begin{prop}   \label{prop:covernilp}
 Let $A$ be a finite group with normal subgroups $G\lhd \tilde G$. Let $p$
 be an odd prime, $B$ a $p$-block of $G$ covered by a nilpotent block of
 $\tilde G$. Suppose that every character of $G$ extends to its inertia group
 in $\tilde G$ and that $\tilde G/G$ is abelian with
 $p\nmid |\w G: G\cent {\tilde G} G|$. Let $\vhi\in \IBr(B)$.
 Then there exists some $p$-rational $A_\vhi$-invariant $\psi\in\Irr(B)$
 with $\psi^0=\vhi$.
\end{prop}

\begin{proof}
Let $\tilde B$ be a nilpotent block of $\tilde G$ covering $B$. Let
$\w\vhi\in\IBr(\w B)$ and $\w \psi\in\Irr(\w B)$ be $p$-rational with
$\w\psi^0=\w\vhi$ from Proposition~\ref{prop:nilpotent}. Now any
$\vhi\in\IBr(B)$ is a constituent of $\w\vhi_G$. As
$p\nmid |\w G: G\cent {\tilde G} G|$, the characters $\w\psi_G$ and $\w\vhi_G$
have the same number of constituents and every constituent of $\w\vhi_G$ lifts
to a unique constituent of $\w\psi_G$.
Let $\psi\in \Irr(B)$ be the constituent of $\w\psi_G$ with $\psi^0=\vhi$.

For every $a\in A_\vhi$ we see that $\w\psi^a$ is a $p$-rational character
of the  block $\tilde B^a$. Note that $\tilde B^a$ is nilpotent as well and
contains a $p$-rational character $\mu\w\psi$ for some linear
$p$-rational character $\mu$ of $\tilde G/G$. Since in a nilpotent block
there exists a unique $p$-rational character, we see $\w\psi^a=\mu\w\psi$.
Then $\psi^a$ is a constituent of $((\w\psi^a)^0)_G=(\mu^0\w\psi^0)_G$.
This implies $\psi^a=\psi$.

Note that since $\w\psi$ is $p$-rational, $\psi$ is $p$-rational
according to the definition.
\end{proof}

To finish this paragraph let us point out nilpotent blocks are not the only
ones containing just one modular character:

\begin{rem}
A typical way of constructing blocks with one modular character which are not
nilpotent is the following. Let $H$ be a $p'$-group of central type (that is,
$H$ is a group possessing $\lambda \in \Irr(\bZ(H))$ such that
$\lambda^H=e\theta$ for some $\theta \in \Irr(H)$ and $e\ge1$). Suppose that
$H/\bZ(H)\ne1$ acts faithfully on a $p$-group $V$. Then the semidirect product
$VH$ has a unique (non-nilpotent) block covering $\lambda$ with a unique
modular character.
\end{rem}

\subsection{Blocks of sporadic groups}

We now start our investigation of $p$-blocks of finite quasi-simple groups
all of whose irreducible Brauer characters lie in a single orbit under the
automorphism group. For groups not of Lie type in cross characteristic we
obtain a full classification.

\begin{prop}   \label{prop:spor}
 Let $G$ be quasi-simple such that $G/\bZ(G)$ is a sporadic simple group or the
 Tits group $\tw2F_4(2)'$.
 Let $p$ be a prime and $B$ a $p$-block of $G$, not of central
 defect, such that $\IBr(B)$ is a single orbit under $\Aut(G)_B$. Then $B$ has
 defect~1 and the degrees of its irreducible Brauer characters are as given
 in Table~\ref{tab:spor}.
\end{prop}

\begin{table}[htbp]
\caption{Blocks in sporadic groups}   \label{tab:spor}
\[\begin{array}{|r|r|r||r|r|r|}
\hline
 G& p& \IBr(B)& G& p& \IBr(B)\\
\hline
      J_1& 2& 76&            Co_3& 2& 129536\\
    2.J_2& 3& 126,126&    Fi_{22}& 2& 2555904\\
   M_{23}& 3& 231&             Ly& 3& 18395586\\
     2.HS& 3& 924, 924&   Fi_{23}& 2& 73531392\\
      McL& 2& 3520&        2.Co_1& 3& 59153976\\
    3.McL\ (2\times)& 2& 6336&     J_4& 3& 786127419\\
\hline
\end{array}\]
\end{table}

\begin{proof}
From the known character tables of covering groups of sporadic simple groups
it is easy to determine the block distribution and the number of modular
Brauer characters in each $p$-block. If $B$ is a block of defect zero, or if
$l(B):=|\IBr(B)|$ is larger than $|\Out(G)_B|$, the block can be discarded.
Also,
for many of the smaller groups, the Brauer character tables are available in
GAP \cite{GAP}. This leaves only the cases in Table~\ref{tab:spor}, and a
5-block $B$ of $Fi_{24}'$ of defect~1 with $l(B)=2$. But the block $B'$ of
$Fi_{24}'.2$ covering $B$ has $l(B')=4$, so this does not give an example.
\end{proof}

Inspection of the known character tables for the groups in Table~\ref{tab:spor}
shows:

\begin{cor}   \label{cor:spor}
 In the situation of Proposition~\ref{prop:spor} every $\vhi\in\IBr(B)$ is
 liftable to an $\Aut(G)_\vhi$-invariant character in $\Irr(G)$, and either
 $\vhi$ has a unique lift, which then is $p$-rational, or $p=3$, there
 are exactly three distinct lifts two of which are algebraically conjugate,
 or $p=2$ and $\vhi$ has exactly two lifts.
\end{cor}

\subsection{Blocks of alternating groups}
We next consider covering groups of alternating groups.

\begin{thm}   \label{thm:alt}
 Let $G$ be a covering group of an alternating group $\fA_n$, $n\ge5$.
 Let $p$ be a prime and $B$ a $p$-block of $G$ of weight $w$,
 not of central defect, such that $\IBr(B)$ is a single orbit under
 $\Aut(G)_B$. Then $B$ has defect one and one of the following occurs:
 \begin{enumerate}[\rm(1)]
  \item $G=\fA_n$, $p=3$, $w=1$, $l(B)=1$ and $B$ is self-conjugate;
  \item $G=2.\fA_n$, $p=3$, $w=1$, $l(B)\le2$;
  \item $G=2.\fA_6$, $p=5$, $w=1$, $l(B)=2$,
   $\{\vhi(1)\mid\vhi\in\IBr(B)\}=\{4\}$; or
  \item $G=6.\fA_6$, $p=5$, $l(B)=2$,
   $\{\vhi(1)\mid\vhi\in\IBr(B)\}=\{6\}$ (two blocks).
 \end{enumerate}
\end{thm}

\begin{proof}
For $n\le7$ we consult the known Brauer character tables. Thus we may
assume that $n\ge8$, the full covering group of $\fA_n$ has center of order~2,
and $|\Out(\fA_n)|=2$.
We use the description of $p$-blocks of alternating groups and their
covering groups in \cite{Ol93}. First consider $B$ a $p$-block of $\fA_n$.
Then $B$ corresponds to a $p$-core partition $\mu$ of an integer
$n-wp$, where $w$ is called the weight of the block. A core and its
conjugate parametrise the same $p$-block. First assume that $p$ is odd.
Then there are two blocks of $\fS_n$ above any non self-conjugate block of
$\fA_n$, and there is one above each self-conjugate one. By
\cite[Prop.~12.8]{Ol93},
$$l(B)=\begin{cases} k(p-1,w)& \text{if $B$ is not self-conjugate},\\
        \frac{1}{2}(k(p-1,w)+3k^s(p-1,w))& \text{else},\end{cases}$$
where $k(a,w)$ denotes the number of $a$-tuples of partions of $w$, and
$k^s(a,w)$ is the number of symmetric $a$-quotients. When $w=0$ then $B$
is of defect zero. For $p=3$ and $w\ge2$ we have
$$l(B)\ge\min\{k(2,2),\frac{1}{2}(k(2,2)+3k^s(2,2))\}
       =\min\{5,4\}=4,$$
which is larger than $|\Out(\fA_n)|$ as $n\ge8$. When $w=1$ and $B$ is
not self-conjugate, then $l(B)=2$, hence this gives no example. For $B$
self-conjugate and $p\ge5$ we find
$$l(B)\ge\min\{k(4,1),\frac{1}{2}(k(4,1)+3k^s(4,1))\}
       =\min\{4,2\}=2,$$
and equality only holds when $p=5$, $w=1$. But in that case the block $B'$
of $\fS_n$ above $B$ has $l(B')=4$ by \cite[Prop.~11.4]{Ol93}. This leaves
only case~(1).
\par
When $p=2$, note that 2-cores are triangular partitions and thus always
self-conjugate. In this case there is a unique 2-block of $\fS_n$ above each
2-block of $\fA_n$. Also, blocks of weight~0 or~1 are of defect~0. By
\cite[Prop.~12.9]{Ol93} we have
$$l(B)=\begin{cases} p(w)& \text{if $w$ is odd},\\
                     p(w)+p(w/2)& \text{if $w$ is even},\end{cases}$$
where $p(w)=k(1,w)$ denotes the number of partitions of $w$. Thus, for $w\ge2$
we have $l(B)\ge3>|\Out(\fA_n)|$, and no further example arises.  \par
Now we turn to spin blocks, that is, faithful blocks of the 2-fold covering
$G=2.\fA_n$, with $p>2$. Any such block $B$ is covered by a unique
$p$-block $B'$ of $2.\fS_n$, and $l(B)=l^n(B')+2l^s(B')$ by
\cite[Prop.~13.19]{Ol93}, where $l^n(B')$, $l^s(B')$ denote the number of pairs
of not self-conjugate Brauer characters in $B$, respectively the number of
self-conjugate ones. A straightforward calculation shows that $l(B)\le2$
implies that either $p=3$, $w\le 2$, or $p=5$, $w=1$, $l(B)=2$. In the
latter case the block $B'$ of $2.\fS_n$, $n\ge8$, above $B$ has $l(B')=4$,
and the same holds if $p=3$ and $w=2$. In the case $p=3$ and $w=1$,
$l(B)=1$ if the sign of the block $B$ is~$+1$, and $l(B)=2$ if it has
sign~$-1$. In either case, $\IBr(B)$ is an orbit under $\Aut(G)$. So we reach
the conclusion in case~(2).
\end{proof}

Note that $2.\fA_6\cong\SL_2(9)$ is a group of Lie type, and the example~(3)
in the preceding result forms part of an infinite series of cases.

\begin{cor}   \label{cor:alt}
 In the situation of Theorem~\ref{thm:alt} all $\vhi\in\IBr(B)$ are
 liftable to an $\Aut(G)_\vhi$-invariant character in $\Irr(G)$, and either
 $\vhi$ has a unique lift, which is $p$-rational, or $p=3$ and there
 are exactly three distinct lifts two of which are algebraically conjugate.
\end{cor}

Note that in all cases of Theorem~\ref{thm:alt}, $p$ is prime to $|\Out(G)|$.

\begin{rem}
The generating function for the number of symmetric 3-cores is known (see
for example \cite[Prop.~9.13]{Ol93}). From this it can be seen that blocks as in
conclusion~(1) of Theorem~\ref{thm:alt} are rather rare; the first few occur
in degrees $n=8,11,19,24$. Similarly, it can be computed from the generating
function in \cite[Prop.~9.9]{Ol93} that the first instances for case~(2) occur
for $n=5,8,10,15,18$, with $l(B)=1$ only when $n=18$.
\end{rem}

\begin{prop}   \label{prop:exc}
 Let $G$ be an exceptional covering group of a finite simple group of Lie
 type. Let $p$ be a prime and $B$ a faithful $p$-block of $G$,
 not of central defect, such that $\IBr(B)$ is a single orbit under
 $\Aut(G)_B$. Then $G=2.G_2(4)$, $p=3$, $B$ has defect~1 and is as in
 Table~\ref{tab:exc}. Every $\vhi\in\IBr(B)$ is liftable to a unique
 $\Aut(G)_\vhi$-invariant character in $\Irr(G)$, which is $3$-rational.
\end{prop}

\begin{table}[htbp]
\caption{Blocks in exceptional covering groups}   \label{tab:exc}
\[\begin{array}{|r|r|r|}
\hline
 G& p& \IBr(B)\\
\hline
       2.G_2(4)& 3& 1800, 1800\\
       2.G_2(4)& 3& 3744, 3744\\
\hline
\end{array}\]
\end{table}

\begin{proof}
As in the proof of Proposition~\ref{prop:spor} this can be checked from the
known ordinary character tables. Apart from the cases listed in the table,
there exist 5-blocks of $12_1.\PSL_3(4)$, of $6_2.\PSU_4(3)$, of $2.\tw2E_6(2)$
and of $6.\tw2E_6(2)$ of defect~1 with two or four irreducible Brauer
characters. But from the shapes of their Brauer trees it is immediate that
they cannot lead to examples.
\end{proof}

\begin{rem}
 For all quasi-simple groups discussed in this section the blocks with one
 orbit of Brauer characters have defect~1. Note that blocks with cyclic
 defect group containing just one modular irreducible character have trivial
 inertial quotient and hence are nilpotent.
\end{rem}

\section{The simple groups of Lie type}   \label{sec:Lie}
In this section we deal with the groups of Lie type by considering them as
subgroups of simple linear algebraic groups. For consistency with the cited
literature we prefer to denote our chosen prime by $\ell$ and denote
by $p$ the characteristic of the underlying field.

We consider the following setup: $\bG$
is a simple linear algebraic group of simply connected type over an algebraic
closure of a finite field $\FF_p$, and $F:\bG\rightarrow\bG$ is a Steinberg
endomorphism with finite group of fixed points $G=\bG^F$. It is well-known
that all quasi-simple groups of Lie type, apart from exceptional covering
groups as considered in Proposition~\ref{prop:exc}, and apart from
$\tw2F_4(2)'$ which was already dealt with in Proposition~\ref{prop:spor},
can be obtained as central quotients of groups $G$ as above. In particular,
any block of a covering group $S$ of a simple group of Lie type is also a
block of such a group $G$.

\subsection{The defining characteristic case}
We first consider the case when $\ell=p$.

\begin{prop}
 Let $S$ be a covering group of a simple group of Lie type, and $\ell=p$
 the defining characteristic of $S$. Let $B$ be an $\ell$-block of $S$, not of
 central defect. Then $\IBr(B)$ is not a single orbit under $\Aut(S)_B$.
\end{prop}

\begin{proof}
All exceptional covering groups $S$ have a center of order divisible by $p$,
so we may assume that $S$ is a non-exceptional covering. In particular, it is
a central quotient of a finite reductive group $G=\bG^F$, with $\bG$ of
simply connected type and $F:\bG\rightarrow\bG$ as above. Denote by
$\delta$ the smallest integer such that $F^\delta$ acts trivially on the Weyl
group of $\bG$ and let $q^\delta$ denote the unique eigenvalue of $F^\delta$
on the character group of an $F$-stable maximal torus of $\bG$. By a result
of Humphreys, the $p$-blocks of $\Irr(G)$ of
non-zero defect are given by $\Irr(G\mid\la)$ for $\la$ running over
$\Irr(\bZ(G))$. Since $p$ divides $|G|$ and $G$ is perfect, the principal
$p$-block of $G$ certainly contains Brauer characters of distinct degrees.
\par
Thus now we may assume that $B$ is a block of $G$ with non-trivial
corresponding central character $\la\in\Irr(\bZ(G))$. Then $\IBr(B)$ is
parametrised by $q^\delta$-restricted $F$-invariant weights of $\bG$ whose
restriction to $\bZ(G)$ is a multiple of $\la$, and $\Aut(G)$ acts on the
characters as it does on the weights. If $\delta=1$, so $F$ is
split, then there are $q^r-1$ weights different from the Steinberg weight
that are $q$-restricted, where $r$ is the rank of $\bG$. Thus, there are at
least two not $\Aut(G)$-conjugate weights for every $\la\in\Irr(\bZ(G))$ if
$q^r-1>2|\bZ(G)|$. It is easily seen that this inequality holds unless $r=1$
and $q\le3$, in which case $G=\SL_2(q)$ is solvable. A very similar discussion
deals with the groups of twisted type. (Note that here we need not consider
the very twisted groups, since their center is trivial.)
\end{proof}

\subsection{Non-defining characteristic}
We now turn to the case of non-defining characteristic.
In our considerations we will make use of the following simple observation
where for any set of characters $X$ of a finite group $G$ we denote by
$X^0:=\{\chi^0\mid\chi\in X\}$ the \emph{multi-set} of restrictions to the
$\ell'$-elements of $G$.

\begin{lem}   \label{lem:orbits}
 Let $B$ be an $\ell$-block of a finite group $G$, and assume that
 $X\subset\ZZ\Irr(B)$ is an $\Aut(G)_B$-invariant subset such that $X^0$ is
 linearly independent. If $\Aut(G)_B$ has at least two orbits on $X$, then
 $\IBr(B)$ is not a single $\Aut(G)_B$-orbit.
\end{lem}

\begin{proof}
By assumption, the $A:=\Aut(G)_B$-permutation module $U:=\CC X^0$ is a
submodule of the $A$-permutation module $V:=\CC\IBr(B)$. Since $A$ is not
transitive on the basis $X^0$ of $U$, the $A$-permutation module $V$
must {\it a fortiori} be intransitive.
\end{proof}

For example this criterion applies if $X$ is an $\Aut(G)_B$-invariant basic set
for the block $B$ which is not a single $\Aut(G)_B$-orbit. The second situation
we will consider is when $X=\{\chi_1,\chi_2\}$, where $\chi_1,\chi_2$ are sums
over $\Aut(G)_B$-orbits of irreducible characters in $\Irr(B)$, such that
$\chi_1^0$ and $\chi_2^0$ are not multiples of one another.

We will also need the following observations pertaining to Lusztig induction:

\begin{prop}   \label{prop:l-rational}
 Let $\bG$ be connected reductive with Frobenius map $F$, and $\ell$ not the
 defining prime for $\bG$.
 \begin{itemize}
  \item[\rm(a)] Let $\bL\le\bG$ be an $F$-stable Levi subgroup. If
   $\la\in\Irr(\bL^F)$ is $\ell$-rational, then so is $\RLG(\la)$.
  \item[\rm(b)] Let $s\in\bG^{*F}$ be a semisimple $\ell'$-element. Then
  $\cE(\bG^F,s)$ is stable under $\ell$-Galois automorphisms.
 \end{itemize}
\end{prop}

\begin{proof}
The character formula \cite[Prop.~12.2]{DM} expresses $\RLG(\la)(g)$,
$g\in\bG^F$, as a linear combination of values $\la(l)$, $l\in\bL^F$, with
integral coefficients (for this note that the 2-variable Green functions are
integral valued by \cite[Def.~12.1]{DM}). Claim~(a) follows. \par
The characters in $\cE(\bG^F,s)$ are by definition the constituents of the
various $\RTG(\theta)$, where $(T,\theta)$ lies in the geometric conjugacy
class of $s$. In particular, all such $\theta$ have $\ell'$-order. The
second assertion is then immediate from part~(a).
\end{proof}

We will show the following slightly more general result than
Theorem~\ref{thm:quasisimples} which is better adapted to our inductive
argument:

\begin{thm}   \label{thm:levis}
 Let $\bG$ be an $F$-stable Levi subgroup of a simple algebraic group $\bH$
 of simply connected type with a Steinberg endomorphism $F:\bH\rightarrow\bH$.
 Let $\ell$ be a prime different from the defining characteristic of $\bH$.
 If $B$ is an $\ell$-block of $G=\bG^F$ such that $\IBr(B)$ is one
 $\Aut(G)_B$-orbit, then any $\vhi\in\IBr(B)$ lifts to some $\ell$-rational
 $\chi\in\cE(G,\ell')$ such that $\chi,\vhi$ have the same stabiliser in
 $\Aut(G)_B$.
\end{thm}

The parts (1)--(3) of Theorem~\ref{thm:quasisimples}
for the quasi-simple groups $S=\bH^F/Z$, with $Z\le \bZ(\bH^F)$ and
$\ell\ne p$ now follow from this by taking $\bG=\bH$.

We will give the proof of Theorem~\ref{thm:levis} in several steps. We begin
by recalling some known facts about $\ell$-blocks of finite reductive groups.
Let $\ell\ne p$ and let $B$ be an $\ell$-block of $G=\bG^F$. By the result of
Brou\'e and Michel (see \cite[Thm.~9.12]{CE}) there exists a semisimple
$\ell'$-element $s$ in the dual group $G^*=\bG^{*F}$ such that $\Irr(B)$ is
contained in the union
$$\cE_\ell(G,s)=\coprod_{t}\cE(G,st)$$
of Lusztig series, where $t$ runs over a system of representatives of
conjugacy classes of $\ell$-elements in $C_{G^*}(s)$. We say that $B$ is
\emph{quasi-isolated} if the centraliser $C_{\bG^*}(s)$ is not contained in any
proper $F$-stable Levi subgroup of $\bG^*$.

Let $\bG_1^*$ be the minimal $F$-stable Levi subgroup of $\bG^*$ (and hence of
$\bH^*$) with $C_{\bG^*}(s)\le\bG_1^*$. Note that this is uniquely determined,
since the intersection of any two Levi subgroups containing $C_{\bG^*}(s)$
(and hence a maximal torus of $\bH^*$) is again a Levi subgroup. Then by
construction $s$ is quasi-isolated in $\bG_1^*$. Let $\bG_1$ be an $F$-stable
Levi subgroup of $\bG$ in duality with $\bG_1^*$ and set $G_1:=\bG_1^F$.
In this situation we have:

\begin{lem}   \label{lem:only qi}
 Theorem~\ref{thm:levis} holds for a block $B$ in $\cE_\ell(G,s)$ if it holds
 for its Jordan correspondent in $\cE_\ell(G_1,s)$.
\end{lem}

\begin{proof}
According to the theorem of Bonnaf\'e and Rouquier \cite[Thm.~10.1]{CE} Lusztig
induction $R_{\bG_1}^\bG$ induces a Morita equivalence between the
$\ell$-blocks in $\cE_\ell(G_1,s)$ and those in $\cE_\ell(G,s)$,
which sends $\Irr(B_1)$ bijectively to $\Irr(B)$ for some $\ell$-block
$B_1$ contained in $\cE_\ell(G_1,s)$. By the very definition of Lusztig
series any automorphism of $G$ either fixes $\cE(G,s)$ or sends it to a
disjoint series $\cE(G,s')$, for $s'\in G^*$ another semisimple
$\ell'$-element. In particular, any automorphism of $G$ stabilising $B$ will
also stabilise $\cE(G,s)$, and thus by the uniqueness of $\bG_1$ will induce
an element of $\Aut(G_1)_{B_1}$. Hence if $\IBr(B)$ is one orbit under
$\Aut(G)_B$, then $\IBr(B_1)$ is one $\Aut(G_1)_{B_1}$-orbit. So if the
assertion of Theorem~\ref{thm:levis} holds for $B_1$, then every
$\vhi_1\in\IBr(B_1)$ lifts to some $\chi_1\in\Irr(B_1)\cap\cE(G_1,\ell')$
with the same stabiliser in $\Aut(G_1)_{B_1}$. Via the Bonnaf\'e--Rouquier
Morita equivalence this shows that any $\vhi\in\IBr(B)$ lifts to some
$\chi\in\Irr(B)\cap\cE(G,\ell')$ with the same stabiliser in $\Aut(G)_B$. As
this bijection is given by $R_{\bG_1}^\bG$ the $\ell$-rationality properties
are also preserved by Proposition~\ref{prop:l-rational}(a).
\end{proof}

\begin{lem}   \label{lem:bs case}
 Theorem~\ref{thm:levis} holds whenever $B$ lies in a Lusztig series
 $\cE_\ell(G,s)$ such that $\cE(G,s)$ is a basic set for $\cE_\ell(G,s)$.
\end{lem}

\begin{proof}
Let $B$ be as in the assertion. By assumption the set $X:=\Irr(B)\cap\cE(G,s)$
is a basic set for $B$, and $\Aut(G)_B$-invariant by the remarks in the proof
of Lemma~\ref{lem:only qi}. Thus, if $\Aut(G)_B$ has more than one orbit on
$X$, then by Lemma~\ref{lem:orbits} it necessarily has more than one orbit on
$\IBr(B)$. So, if $\IBr(B)$ is a single $\Aut(G)_B$-orbit, then all elements
in $X$ are $\Aut(G)_B$-conjugate and in particular have the same degree $d$,
say. But since $X$ is a basic set, then all elements of $\IBr(B)$ must have
this degree~$d$, hence $\IBr(B)$ consists of the Brauer characters of the
$\ell$-modular reductions of the characters in $X$. In particular, they are
all liftable to elements of $\cE(G,\ell')$ under preservation of stabiliser
in $\Aut(G)_B$, and these lifts have to be $\ell$-rational by
Proposition~\ref{prop:l-rational}.
\end{proof}

Before we continue, we need to recall an auxiliary result about the existence
of ordinary basic sets:

\begin{lem}   \label{lem:basic set}
 Let $\bG$ be connected reductive with a Steinberg endomorphism $F$. Let
 $s\in\bG^{*F}$ be a semisimple $\ell'$-element, and $\bG_1\le\bG$ an
 $F$-stable Levi subgroup such that $\bG_1^*$ contains $C_{\bG^*}(s)^F$.
 Assume that $\ell$ is good for $\bG_1$ and prime to the order of
 $(\bZ(\bG_1)/\bZ^\circ(\bG_1))^F$.
 Then $\cE(\bG^F,s)$ is a basic set for the blocks in $\cE_\ell(\bG^F,s)$.
\end{lem}

\begin{proof}
By the theorem of Geck and Hiss \cite[Thm.~14.4]{CE} the assumptions made on
$\bG_1$ ensure that $\cE(\bG_1^F,s)$ is a basic set for the blocks in
$\cE_\ell(\bG_1^F,s)$. Now by \cite[Thm.~10.1]{CE}
Lusztig induction $R_{\bG_1}^\bG$ induces a Morita equivalence between the
$\ell$-blocks in $\cE_\ell(\bG_1^F,s)$ and those in $\cE_\ell(\bG^F,s)$,
which sends $\cE(\bG_1^F,s)$ bijectively to $\cE(\bG^F,s)$. The claim follows.
\end{proof}

\begin{cor}   \label{cor:basic set}
 Let $\bG$ be simple of type $G_2$, $F_4$ or $E_6$ and $\ell=3$, or of type
 $E_8$ and $\ell=5$. Then $\cE(\bG^F,s)$ is a basic set for the $\ell$-blocks
 in $\cE_\ell(\bG^F,s)$ unless $s\in\bG^{*F}$ is quasi-isolated (and hence of
 order at most~2 if $\bG$ does not have type $E_8$).
\end{cor}

\begin{proof}
In all listed cases, the prime $\ell$ is good for any proper Levi subgroup
of $\bG$. Unless $\bG$ is of type $E_6$, $\bG$ has trivial center, so all
assumptions of Lemma~\ref{lem:basic set} are satisfied as soon as
$C_{\bG^*}(s)^F$ is contained in a proper Levi subgroup of $\bG$.
In type $E_6$ it is readily checked that any proper Levi subgroup
has connected center.
\end{proof}

\begin{prop}   \label{prop:type A}
 Theorem~\ref{thm:levis} holds when $\bH$ is of type $A$.
\end{prop}

\begin{proof}
Note that any prime $\ell$ is good for $\bH$, hence for $\bG$. If $\ell$ does
not divide $|\bZ(\bG)^F|$ then $\cE(G,s)$ is a basic set for all blocks in
$\cE_\ell(G,s)$ by Lemma~\ref{lem:basic set}, and the claim follows from
Lemma~\ref{lem:bs case}.
\par
Now first assume that $\bH^F=\SL_n(q)$. As $|\bZ(\bH)^F|=\gcd(n,q-1)$ we may
assume that $\ell|(q-1)$. Let $B$ be an $\ell$-block of $G=\bG^F$ in
$\cE_\ell(G,s)$. Then by \cite[Thm.~A(a)]{KM15} there exists a 1-cuspidal
pair $(\bL,\la)$ with a 1-split Levi subgroup $\bL\le \bG$ such $\Irr(B)$
contains all constituents of $\RLG(\la)$. Consider a regular embedding
$\bG\hookrightarrow\tilde\bG$ and let $\tilde s\in \tilde \bG^{*F}$ be
a preimage of $s$ under the induced map $\tilde\bG^*\rightarrow\bG^*$ of dual
groups. Then $\cE(\tilde G,\tilde s)$ is a single 1-Harish-Chandra series, for
a 1-cuspidal character $\tla$ lying above $\la$ of the Levi subgroup $\tilde L$
of $\tilde G$ with $\tilde L\cap G=\bL^F$. Now let $\chi_1,\chi_2$ be
constituents of $\RLG(\la)$ below the semisimple and the regular character of
$\cE(\tilde G,\tilde s)$, respectively.
\par
First assume that $\chi_1(1)\ne\chi_2(1)$. As $\ell$ is good for $G$,
$\cE(G,s)^0$ is linearly independent by \cite[Thm.~14.6]{CE}, hence the
$\Aut(G)_B$-orbit sums of $\chi_1,\chi_2$ satisfy the assumptions of
Lemma~\ref{lem:orbits}, and $\IBr(B)$ cannot be a single $\Aut(G)_B$-orbit.
Thus we have $\chi_1(1)=\chi_2(1)$. But these correspond, under Jordan
decomposition, to the trivial and the Steinberg character of
$C_{\bG^*}^\circ(s)^F$,
whence $C_{\bG^*}^\circ(s)$ is a torus. Then so is $C_{\tilde\bG^*}(\tilde s)$,
and hence $\cE(\tilde G,\tilde s)=\{\tchi\}$ has just one element, and since
$\tilde\bG$ has connected center and $\ell$ is good for $\tilde\bG$,
$\{\tchi\}$ is a basic set for the block $\tilde B$ of $\tilde G$
covering $B$ containing $\tchi$. In particular $\IBr(\tilde B)=\{\tchi^0\}$,
and $\tchi,\tchi^0$ are invariant under the same automorphisms of
$\tilde G$.
\par
Since $\cE(G,s)^0$ is linearly independent (again by \cite[Thm.~14.6]{CE}),
all elements of $\Irr(B)\cap\cE(G,s)$ must have the same degree $d$, say.
Now clearly $C_{\bG^*}(st)\le C_{\bG^*}(s)$ for all $\ell$-elements
$t\in C_{\bG^*}(s)$, so all characters in $\cE_\ell(G,s)$ have degree
divisible by $d$ by the Jordan decomposition of characters. Thus all elements
of $\IBr(B)$ have degree divisible
by $d$. By what we said before, then certainly $\cE(G,s)^0\subseteq\IBr(B)$.
On the other hand, since $\cE(\tilde G,\tilde s)=\{\tchi\}$ and
$\IBr(\tilde B)=\{\tchi^0\}$, we must actually have $\cE(G,s)^0=\IBr(B)$.
Then clearly $B$ satisfies the conclusion of Theorem~\ref{thm:levis}.
\par
If $\bH^F=\SU_n(q)$, then necessarily $\ell$ divides $q+1$, and we can argue
entirely similar, with $1$-cuspidal and $1$-series replaced by $2$-cuspidal
and $2$-series.
\end{proof}

\begin{prop}   \label{prop:classical}
 Theorem~\ref{thm:levis} holds if all components of $\bG$ are of classical
 type $A$, $B$, $C$ or~$D$ and moreover $\ell\ne2$ if $\bG^F$ has a component
 of type $\tw3D_4$.
\end{prop}

\begin{proof}
By Lemma~\ref{prop:type A} we may assume that $\bH$ is not of type $A$.
We claim that $\bZ(\bG)/\bZ^\circ(\bG)$ is of 2-power order. Indeed, if $\bH$
is not of type $E_6$, then $\bZ(\bH)$ has 2-power order and the claim follows.
If $\bH$ is of type $E_6$, then it is easily verified that all proper Levi
subgroups have connected center (in fact, this is only an issue for those
having a factor $A_5$ or $A_2$).
\par
Now first assume that $\ell>2$. Let $B$ be an $\ell$-block of $G$, lying in
series $s\in \bG^{*F}$. By our previous observation
$|\bZ(\bG)^F/\bZ^\circ(\bG)^F|$ is prime to $\ell$, and as $\ell\ge3$ is good
for $\bG$,
Lemma~\ref{lem:basic set} applies to $B$ to yield that $\Irr(B)\cap\cE(G,s)$
is a basic set for $B$. Thus, we conclude by Lemma~\ref{lem:bs case}.
\par
We are left to consider the case $\ell=2$. Let $s\in G^*$ be a semisimple
$2'$-element such that $\Irr(B)\subset\cE_\ell(G,s)$. Recall that $G$ has no
component of type $\tw3D_4$. Then by a result of Enguehard (see
e.g.~\cite[Lemma~3.3]{KM15}) we have that $\Irr(B)=\cE_\ell(G,s)$. Since
$\bZ(\bG)/\bZ^\circ(\bG)$ is a 2-group and $s$ has odd order, the centraliser
$C_{\bG^*}(s)$ is connected and a Levi subgroup of $\bG^*$. Let $\bG_1\le\bG$
be an $F$-stable Levi subgroup dual to $C_{\bG^*}(s)$. By
Lemma~\ref{lem:only qi} we may pass to the Morita equivalent Jordan
corresponding block $B_1$ in $\cE(\bG_1^F,s)$, which in turn is Morita
equivalent to the principal block of $\bG_1^F$ as $s\in\bZ(\bG_1^{*F})$.
The latter obviously satisfies the assertions of Theorem~\ref{thm:levis}.
\end{proof}

\begin{prop}   \label{prop:exc type}
 Theorem~\ref{thm:levis} holds when $\bG^F$ has a component of exceptional
 type.
\end{prop}

\begin{proof}
By Lemma~\ref{lem:bs case} in conjunction with Lemma~\ref{lem:basic set} we
may assume that $\ell$ is bad for $\bG$. Moreover, by Lemma~\ref{lem:only qi}
the block $B$ is quasi-isolated in $G$. There is no bad non-defining
prime for Suzuki groups. For the Ree groups $^2G_2(3^{2n+1})$ the only relevant
bad prime is $\ell=2$, and the only quasi-isolated block is the principal
block by Ward \cite{Wa66}, so we are done in this case. The only bad
prime for $\tw2F_4(2^{2n+1})$ is $\ell=3$, and $s=1$ is the only quasi-isolated
semisimple $3'$-element. Here the unipotent blocks were determined in
\cite[Bem.~1]{MaF} and the tables in loc.\ cit.\ show that in each such block
of positive defect there are two $\Aut(G)$-invariant unipotent characters with
linearly independent restriction to $\ell'$-elements, so Lemma~\ref{lem:orbits}
allows to conclude.
\par
We are left to consider the exceptional groups of types $G_2$, $\tw3D_4$,
$F_4$, $E_6$, $\tw2E_6$, $E_7$ and $E_8$. Assume first that $\bG=\bH$. For
these groups the unipotent blocks at bad primes have been determined by
Hiss--Shamash, Deriziotis--Michler and Enguehard (see the tables in
\cite{En00}), while their other
quasi-isolated blocks can be found in Kessar--Malle \cite[Tables~1--8]{KM13}.
It is easily checked from these sources that each such block of non-zero defect
contains two characters $\chi_1,\chi_2$ with different $a$-values
$a_i=a(\chi_i)$ (in the sense of \cite[\S3B]{GM99}), with $a_1<a_2$ say.
It then follows by
\cite[Thm.~3.7]{GM99} that there exists an $F$-stable unipotent conjugacy class
$\bC\subset\bG$ such that the variety of Borel subgroups of $\bG$ containing a
given $u\in \bC$ has dimension~$a_1$ with the following property: the average
value of $\chi_2$ on $\bC^F$ is zero, while it is non-zero for $\chi_1$. Since
the unipotent classes of $\bG$ are invariant under all automorphisms of $\bG$,
and thus their $F$-fixed points $\bC^F$ are invariant under all automorphisms
of $G$, the orbit sums $\psi_1,\psi_2$ of $\chi_1$ and $\chi_2$ under
$\Aut(G)_B$ still have the same vanishing respectively non-vanishing property
on $\bC^F$. In particular $\psi_1^0,\psi_2^0$ are linearly independent. Thus,
by Lemma~\ref{lem:orbits} there are at least two orbits of $\Aut(G)_B$ on
$\IBr(B)$. (For example, in many cases, we can take for $\chi_1$ the semisimple
character and for $\chi_2$ the regular character in $\cE(G,s)$.)
\par
Finally assume that $[\bG,\bG]$ is not simple, but has some exceptional
component
or a component with $F$-fixed points of type $\tw3D_4$. In the latter case,
since $\tw3D_4(q)$ has trivial Schur multiplier, $G$ splits into a direct
product, and since the claim holds for the factors, it clearly also holds
for $G$. The only possibility in the former case is that $\bG$ has type
$E_6+A_1$ inside $E_8$, and $\ell\in\{2,3\}$. Note that then $\bG$ has connected
center, so the quasi-isolated elements are just the isolated elements in the
component of type $E_6$. For those we had already argued that all $\ell$-blocks
$B$ not of defect zero contain at least two $\Aut(G)_B$-orbits on $\Irr(B)$
with linearly independent restrictions to $\ell'$-classes, so we can conclude
as before.
\end{proof}

Together, Propositions~\ref{prop:type A}--\ref{prop:exc type} cover all cases
and thus the proof of Theorem~\ref{thm:levis} is complete.
Note that here we find examples of blocks of arbitrarily high defect. Indeed,
if $s\in G^*$ is a regular semisimple $\ell'$-element lying in the unique
maximal torus $T^*$ then the block containing the corresponding
Deligne--Lusztig character is covered by a nilpotent block of $\tilde G$ and
has defect at least $|T^*|_\ell$, which can be an arbitarily high power of
$\ell$.

\subsection{Extendibility}
We now turn to the proof of Theorem~\ref{thm:quasisimples}(4) in our situation
and assume in the following that $\ell=2$.
Let $\bG$ be simple, simply connected, and $\bG\hookrightarrow\tilde\bG$ a
regular embedding, so $\tilde\bG$ has connected center and
$\tilde\bG=\bZ(\tilde\bG)\bG$, with compatible Frobenius map $F$. Recall that
the automorphism group of $G:=\bG^F$ has a subgroup $\Diag(G)$ consisting of
automorphism induced by elements of $\tilde G:=\tilde\bG^F$

\begin{prop}   \label{prop:does extend}
 Let $G,B,\chi$ be as in Theorem~\ref{thm:levis}, $\ell=2$, and let $P$ be a
 Sylow $2$-subgroup of $\Aut(G)_\chi$. If the Sylow $2$-subgroups of
 $\Aut(G)_{\chi}/\Diag(G)_\chi$ are cyclic (which is the case whenever
 $G\notin \{\SL_n(q), D_n(q),E_6(q)\}$),
 then $\chi$ extends to some $\tilde \chi\in \Irr(G\rtimes P)$ such that
 $\ker(\chi_{\cent {G\rtimes P} G})$ has odd index in $\cent{G\rtimes P} G$.
\end{prop}

\begin{proof}
Let $A$ be a Sylow $2$-subgroup
of $\tilde \bG^F$. Now by \cite[Prop.~1.3]{CE99}, restriction induces a
bijection between $\cE(GA,2')$ and $\cE(G,2')$. Hence $\chi^{GA}$ has
a unique constituent $\tilde\chi$ in $\cE(GA,2')$, which then must be
$\Aut(G)_\chi$-stable as well, with respect to the natural action of $\Aut(G)$
on $GA$. By construction $\ker(\tilde \chi)$ contains the Sylow
$2$-subgroup
of $\bZ(GA)$. Any Sylow $2$-subgroup $D$ of $\Aut(G)_\chi/\Inn(GA)$ is
isomorphic to a group generated by field and graph automorphisms of $G$ that
act faithfully on $GA$. By our assumption on $\Aut(G)_\chi/\Diag(G)_\chi$,
$D$ is cyclic. Hence $\tilde\chi$ extends further to $(GA)D$.
By our construction $\cent{(GA)D}G$ is contained in $\bZ(\tilde G)$ and the
$2$-part $\tilde Z$ of $\bZ(\tilde G)$ is contained in $\ker(\tilde\chi)$.

In order to now prove the extendibility of $\chi$ to $G\rtimes P$ we apply
Lemma \ref{extendingclean}: Let $\overline G:=G/\bZ(G)_2$, where
$\bZ(G)_2$ is the Sylow $2$-subgroup of $\bZ(G)$.
Then the above implies that $\chi$ seen as a character of $\overline G$
extends to $(GA)D/\tilde Z$. Now the group $\cent {(GA)D/\tilde Z}{\overline G}$
is an $2'$-group, and Lemma~\ref{extendingclean} proves that
$\chi$ as a character of $\overline G$ extends to $\overline G\rtimes P$.
Since  $\Aut(G)=\Aut({\overline G})$, $P$ naturally acts on $G$ and
this implies the extendibility part of the statement.

Note that since $D$ is isomorphic to a group generated by field and graph
automorphisms, $D$ is cyclic if $G\notin\{\SL_n(q),D_n(q),E_6(q)\}$.
\end{proof}

It remains to deal with the small prime cases left open in
Proposition~\ref{prop:does extend}.

\begin{prop}   \label{prop:typeA}
 Let $G,B,\chi$ be as in Theorem~\ref{thm:levis}. Assume that $\ell=2$ and
 $G=\SL_n(q)$ for some $n\ge3$ and some odd prime power $q$. Let $P$ be a
 Sylow $2$-subgroup of $\Aut(G)_\chi$. Then $\chi$ extends to some character
 $\tilde\chi$ of $G\rtimes P$ such that
 $\ker(\tilde \chi|_{\cent{G\rtimes P}{G}})$ has odd index in
 $\cent{G\rtimes P}{G}$.
\end{prop}

\begin{proof}
By Proposition~\ref{prop:does extend} we may assume that
$\Aut(G)_\chi/\Diag(G)_\chi$ is not cyclic. Hence we can assume that
the $\tilde \bG^F$-orbit of $\chi$ is invariant under the graph automorphism
$\gamma$ and a field automorphism of even order. Let $D_2$ be a Sylow
$2$-subgroup of the stabiliser of this orbit in the group generated by
field automorphisms and $\gamma$. Let $F_0$ be a field automorphism such that
$D_2=\langle F_0,\gamma\rangle$. According to \cite[Thm.~4.1]{CS15} some
character in the $\tilde \bG^F$-orbit of $\chi$ is $D_2$-invariant. We may
assume that $\chi$ is $D_2$-invariant. The character
$\tilde\chi\in\cE(GA,2')$ from Proposition~\ref{prop:does extend}
extends to $\tilde \bG^F_\chi$ and is $D_2$-invariant. Since the quotient
$\tilde \bG^F_\chi/GA$ has odd order there exists a $D_2$-invariant extension
$\widehat \chi$ of $\tilde \chi$ to $\tilde\bG^F_\chi$ according to (some easy
application of) \cite[Lemma~13.8]{Isa}.
Now $\widehat \chi^{\tilde \bG^F}$ extends to $\tilde \bG^F D_2$, since by
\cite[Lemma~4.3.2]{Cedric} there exists an extension $\psi$ of
$\widehat \chi^{\tilde \bG^F}$ to $\tilde \bG^F \langle F_0 \rangle$ such
that $\psi(F_0)\neq 0$. Accordingly $\psi^\gamma=\psi$  and $\psi$ extends to
$\tilde \bG^F\langle F_0,\gamma \rangle$.
By considering the degrees this proves that $\chi$ extends to some
$\tilde \chi\in \Irr((\tilde \bG^F D_2)_\chi)$ such that $\ker(\tilde \chi)$
contains a Sylow $2$-subgroup $\tilde Z$ of $\cent{(\tilde \bG^F D_2)_\chi}G$.

Now applying Lemma \ref{extendingclean} as in the proof of
Proposition~\ref{prop:does extend} implies the extensibility part of the
statement.
\end{proof}

Alternatively, the previous result could be proved by using Deligne--Lusztig
characters for disconnected groups as introduced by Digne--Michel.

\begin{prop}
 Let $G,B,\chi$ be as in Theorem~\ref{thm:levis}.  Assume that $\ell=2$ and
 $G\in\{D_n(q),E_6(q)\}$. Let $P$ be a Sylow $2$-subgroup of $\Aut(G)_\chi$.
 Then $\chi$ extends to some $\tilde \chi\in \Irr(G\rtimes P)$ such that
 $\ker(\widetilde\chi|_{\cent {G\rtimes P}G})$ has odd index in
 $\cent{G\rtimes P}G$.
\end{prop}

\begin{proof}
First assume that $\bG$ has type $D_n$ with untwisted Frobenius map.
Let $s\in G^*$ be a semisimple $2'$-element with $\Irr(B)\subseteq\cE_2(G,s)$.
Then as shown in the proof of Proposition~\ref{prop:classical}, $B$ is
Morita equivalent to the principal block $B_1$ of $\bG_1$, where $\bG_1$ is
dual to the Levi subgroup $C_{\bG^*}(s)$, and
$\cE(\bG_1^F,1)\subseteq\Irr(B_1)$. Thus, $B_1$ contains all unipotent
characters of $\bG_1^F$, hence it will satisfy our assumptions only if the
trivial character is the only unipotent character of $\bG_1^F$, hence if
$\bG_1$ is a (maximal) torus.
Since $\bG_1$ is abelian, every block of $\bG_1^F$ is nilpotent. Accordingly
the block $B$ is nilpotent and the statement follows from
Proposition~\ref{prop:nilpotent}.
\par
Now let $\bG$ be of type $E_6$ with untwisted Frobenius map and assume that
Proposition \ref{prop:does extend} does not apply.
This forces $\Aut(G)_B$ to contain graph and field automorphisms of order~2.
Let $s\in G^*$ be a semisimple $2'$-element with $\Irr(B)\subseteq\cE_2(G,s)$.
If $C_{\bG^*}(s)$ is contained in a proper $F$-stable Levi subgroup $\bG_1^*$,
then all components of $\bG_1$ are of classical type, the center of $\bG_1$
is connected, and thus $\Irr(B)=\cE_2(G,s)$. We see that then our assumptions
imply that $C_{\bG^*}(s)$ is a torus. This implies again that $B$ is nilpotent,
and the statement follows from Proposition~\ref{prop:nilpotent}.  \par
Thus we may assume that $s$ is quasi-isolated, hence as in \cite[Tab.~3]{KM13}.
Inspection shows that unless $B$ is of defect~0 there always exist characters
with distinct $a$-values in $\Irr(B)\cap\cE(G,s)$, and hence $B$ does not
satisfy our assumptions.
\end{proof}

This completes the proof of Theorem~\ref{thm:quasisimples}.


\section{$p$-rational lifts}   \label{prational}

From now we (again) consider $p$-blocks or $p$-Brauer characters for some
odd prime $p$. In this section we show in Theorem \ref{thm:prationalisom}
that for an odd prime $p$ in some situations $p$-rationality is preserved by
isomorphisms of character triples. This ensures that in the proof of the
Main Theorem the lift can be chosen $p$-rational.

Recall that $\chi \in \Irr(G)$ is \emph{$p$-rational} if the values of $\chi$
lie in some cyclotomic field $\QQ_m:=\QQ(\zeta)$, for some $m$th root of
unity $\zeta$ of order not divisible by $p$.
By elementary character theory, this happens if the values of $\chi$ lie in
$\QQ_{|G|_{p'}}$. If $\chi$ is $p$-rational, it is non-trivial to prove that
$\chi$ can be afforded by an absolutely irreducible
$\QQ_m$-representation (see \cite[Thm.~(10.13)]{Isa}.)

\medskip
If $N\nor G$ and $\theta\in\Irr(N)$ is $p$-rational it is not true in
general that there exists a $p$-rational irreducible character
$\chi \in \Irr(G)$ over $\theta$, even if $G/N$ is a $p'$-group. We start
with the following.

\begin{thm}   \label{thm:overpratp'}
 Suppose that $N \nor G$ with $G/N$ a $p'$-group for $p$ odd. Assume
 that $\theta \in \Irr(N)$ is $p$-rational such that $\theta^0 \in \IBr(N)$.
 Then every $\chi \in \Irr(G)$ over $\theta$ is $p$-rational.
\end{thm}

\begin{proof}
We may assume that $\theta$ is $G$-invariant.
Let $\chi \in \irr{G|\theta}$ and $x \in G$. Let $m=|G|_{p'}$. We want to show
that $\chi(x) \in \QQ_m$. We may assume that $G/N$ is generated by $N$ and
$x$. Hence $\chi_N=\theta$. Now let $\sigma \in \Gal(\QQ_{|G|}/\QQ_m)$.
We have that $\chi^\sigma=\lambda \chi$, for some $\lambda \in \irr{G/N}$.
Now, $(\chi_N)^0=\theta^0 \in \IBr(N)$. Thus $\chi^0=\vhi \in \IBr(G)$,
and $\vhi_N=\theta^0$. Now, $\vhi^\sigma=\vhi$. Thus $(\chi^\sigma)^0=\chi^0$.
Hence $\lambda^0\chi^0=\chi^0$. By Gallagher's Lemma for Brauer characters,
$\lambda^0=1$, and therefore $\lambda=1$.
\end{proof}

\begin{cor}   \label{cor:overprat}
 Suppose that $G/N$ has a normal Sylow $p$-subgroup $P/N$. Let
 $\theta \in \Irr(N)$ be $p$-rational such that $\theta^0 \in \IBr(N)$.
 Then there exists a $p$-rational $\chi\in\irr{G|\theta}$.
\end{cor}

\begin{proof}
We may assume that $\theta$ is $G$-invariant. Now, $\theta$ has a canonical
$p$-rational extension $\hat\theta \in \irr{P|\theta}$ by
\cite[Thm.~(6.30)]{Isa}. Since it extends, it lifts an irreducible Brauer
character. Now apply Theorem~\ref{thm:overpratp'} to $\hat\theta$.
\end{proof}

\begin{cor}   \label{cor:pratovercyc}
 Let $N$ be a normal subgroup of $G$ with $G/N$ cyclic and let
 $\theta\in\Irr(N)$ be a $G$-invariant $p$-rational character of $N$ with
 $\theta^0\in\IBr(N)$. Then there exists a $p$-rational character
 $\chi\in\Irr(G)$ extending $\theta$.
\end{cor}

\begin{proof}
This follows from Corollary~\ref{cor:overprat}.
\end{proof}

In the following we apply a standard construction of \emph{ordinary-modular
character triples} as developed in \cite{N} taking additionally into account
ordinary and modular characters at the same time as well as fields of values.
If $N \nor G$ and $\theta\in\Irr(N)$ is $G$-invariant,
then we say that $(G,N, \theta)$ is a \emph{character triple}.
The classical theory of projective representations allows us to replace
$(G,N, \theta)$ by another triple $(\Gamma,M, \psi)$ in which $M$ is in the
center of $\Gamma$, and the character theory of $G$ over $\theta$ is
\emph{analogous} to the character theory of $\Gamma$ over $\psi$.
(See Chapter 11 of \cite{Isa} for details.)
Furthermore, if $\theta^0 \in \IBr(N)$, then this replacement can be done
in such a way that both the ordinary and the modular character theory
are preserved. (See Problem (8.12) of \cite{N}.)

\medskip

Recall that for the definition of irreducible Brauer characters a maximal
ideal $M$ of the ring of algebraic integers $\bR$ in
$\CC$ has been chosen, with respect to which Brauer characters are calculated.
As in \cite[Sect.~2]{N}, let $F={\bR}/M$ be an algebraically closed field of
characteristic $p$, and $^*:{\bR} \rightarrow F$ be the canonical
homomorphism. Note that $^*$ can be extended to the localisation $S$
of ${\bR}$ at $M$, and therefore induces a map $\Mat_n(S)
\rightarrow \Mat_n(F)$, denoted again by $^*$ .

\begin{lem}   \label{Fproj}
 Suppose that $N\nor G$ and let $\theta\in\Irr(N)$ be a $G$-invariant
 $p$-rational character with $\theta^0\in\IBr(N)$, where $p$ is odd.
 Let $\FF=\QQ_m$, where $m=|G|_{p'}$.  Then there exists a projective
 representation $\cX$ of $G$ with entries in $\FF \cap S$ and with
 factor set $\alpha$ satisfying the following conditions:
 \begin{enumerate}[\rm(1)]
  \item $\cX_N$ is an ordinary representation of $N$ affording $\theta$;
  \item $\cX(gn)=\cX(g)\cX(n)$, $\cX(ng)=\cX(n) \cX(g)$ for $g\in G$
   and $n\in N$;
  \item $\alpha(g,h)\in \FF$; and
  \item $\cX^*$ defined as $\cX^*(g)=\cX(g)^*$ for $g \in G$, is an
   $F$-projective representation such that $(\cX^*)_N$ affords the Brauer
   character $\theta^0$.
 \end{enumerate}
\end{lem}

\begin{proof}
By  \cite[Thm.~(10.13)]{Isa} and Problem (2.12) of \cite{N} we know that
$\theta$ can be afforded by an absolutely irreducible $\FF N$-representation
$\cY$ with entries in $\FF \cap S$. Hence $\cY^*$ affords $\theta^0$.
For each $\bar g =gN\in G/N$, by using Corollary~\ref{cor:pratovercyc}, there
is a $p$-rational character $\psi$ of $N\langle g\rangle=H$ extending $\theta$.
Again $\psi$ can be afforded by an absolutely irreducible
$\FF H$-representation $\cY_1$ with entries in $\FF \cap S$. Notice that
$\cY_1^*$ is irreducible since it affords the irreducible Brauer character
$\psi^0$ (which extends $\theta^0$). We would like to choose
$\cY_1$ such that it extends $\cY$. We know that $(\cY_1)_N$ and $\cY_N$ are
$\FF$-similar. By using \cite[Lemma (2.5)]{N} there exists a matrix $P$ with
entries in $\FF \cap S$ such that $(\cY_1)_N=P^{-1} \cY P$, and $P^* \ne 0$.
Thus $P^*((\cY_1)^*)_N=\cY^* P^*$, and so by Schur's Lemma
$P^*$ is invertible. Therefore $P^{-1}$ has entries in $\FF \cap S$.
Hence, $\cY_{\bar g}:=P\cY_1P^{-1}$ is an $\FF$-representation of
$\langle N,g \rangle$ with entries in $\FF \cap S$ extending $\cY$. We define
$$\cX(g)=\cY_{\bar g} (g)\qquad\text{for $g\in G$}.$$
By \cite[Lemma (2.2)]{Isa0}, note that $\cX$ is a projective representation
of $G$, with factor set $\alpha$, say. Also, $\cX(gn)=\cX(g) \cY(n)$,
$\cX(ng)=\cY(n) \cX(g)$ and $\cX(n)=\cY(n)$ for $n \in N$ and $g \in G$.
Furthermore the factor set $\alpha$ of $\cY$ satisfies
$$\alpha(g,h) \in \FF\quad \text{for all $g, h \in G$} $$
by definition. Condition~(4) now easily follows from the definition of $\cX$.
\end{proof}

Suppose that $(G,N, \theta)$ is an ordinary character triple with
$\theta^0 \in \IBr(N)$. Then we say that $(G,N, \theta)$ is an
\emph{ordinary-modular} character triple. The precise definition of
isomorphism of ordinary-modular character triples is given in Problem (8.10)
of \cite{N}.

\begin{thm}   \label{thm:prationalisom}
 Suppose that $(G,N,\theta)$ is a character triple where $\theta$ is a
 $p$-rational character, with $p$ odd, and $\theta^0 \in \IBr(N)$. Then there
 exists an isomorphic ordinary-modular character triple $(\Gamma, M, \xi)$
 satisfying the following conditions.
 \begin{enumerate}[\rm(a)]
  \item $M$ is a central $p'$-subgroup of $\Gamma$.
  \item If $N \subseteq U \subseteq G$ and $\chi\in \irr{U \mid\theta}$,
   then $\chi$ is $p$-rational if and only if its correspondent $\tilde\chi$
   is $p$-rational.
 \end{enumerate}
\end{thm}

\begin{proof}
We  follow the proof of \cite[Thm.~(8.28)]{N}. Let $\cX$ be the projective
representation of $G$ with factor set $\alpha$ determined by
Corollary~\ref{Fproj}. Then $\cX$ is a projective representation of $G$
associated with $\theta$ satisfying conditions~(a)--(c) of
Theorem (11.2) of \cite{Isa}. In particular, $\alpha(g,n)=\alpha(n,g)=1$
for $n \in N$ and $g \in G$ and $\alpha(gn,hm)=\alpha(g,h)$ for $g,h \in G$
and $n,m \in N$ (see pages 178 and 179 of \cite{Isa}).

Now, let $E$ be the group of $|G|_{p'}$-th roots of unity contained in
$\FF$. Notice that $E$ is a $p'$-group.
Let $\tilde G=\{(g,\epsilon)\mid g\in G, \epsilon\in E\}$ with multiplication
$$(g_1, \epsilon_1)(g_2, \epsilon_2)=
  (g_1g_2, \alpha(g_1,g_2)\epsilon_1 \epsilon_2) \, .$$
The fact that $\alpha $ is in $\bZ^2(G, \CC)$ makes this
multiplication associative. Since $\alpha(g,1)=\alpha(1,g)=\alpha(1,1)=1$
and $\alpha(g,g^{-1})= \alpha(g^{-1}, g)$ for $g \in G$ (by using that
$\alpha$ is a factor set), it  easily follows that $(1,1)$ is the identity
of $\tilde G$  and that $(g^{-1}, \alpha(g,g^{-1})^{-1}\epsilon^{-1})$ is the
inverse of $(g, \epsilon)$ for $g \in G$ and $\epsilon \in E$.
Hence, $\tilde G$ is a finite group.

Let $\tilde N=\{(n,\epsilon) \mid n\in N, \epsilon\in E \}\subseteq\tilde G$.
Note that
$$(n_1, \epsilon_1)(n_2, \epsilon_2)= (n_1n_2, \epsilon_1\epsilon_2)$$
since $\alpha(n_1,n_2)=1$ for $n_1,n_2 \in N$. Hence, $\tilde N=N \times E$
is a subgroup of $\tilde G$. To make the notation easier we identify $E$
with $1 \times E$ and $N$ with $N \times 1$ whenever we find it necessary.
If $n \in N$ and $\delta \in E$, then, by using the fact that
$\alpha $ is constant on cosets modulo $N$, we easily check that
$$(g, \epsilon)(n, \delta)(g, \epsilon)^{-1}= (gng^{-1}, \delta)\qquad
  \text{for $(g,\epsilon)\in\tilde G$}.$$
Thus $E$ is a subgroup of $\bZ(\tilde G)$ and $N$ and $\tilde N$
are normal subgroups of $\tilde G$. Define
$$\tilde \cX(g, \epsilon):=\epsilon\cX(g) \, .$$
It is clear that $\tilde \cX$ is an $\FF$-representation of $\tilde G$.
Call $\tau$ its $p$-rational character. Note that
$$\tilde \cX(n, \epsilon)= \epsilon\cY(n)\qquad\text{for $n\in N$},$$
and so
$$\tau(n, \epsilon)=\epsilon\theta(n) \, .$$
Now, define $\tilde \theta=\theta \times 1_{E} \in \irr{\tilde N}$ and observe
that $\tilde \theta_N=\theta=\tau_N \in \Irr(N)$. In particular,
$\tau \in \irr{\tilde G}$ and $\tau^0 \in \IBr(\tilde G)$. Also, note that
$\tilde \theta$ is $\tilde G$-invariant because $\theta$ is $G$-invariant.

The map $\tilde G \rightarrow G$, $(g, \epsilon) \mapsto g$, is an 
epimorphism with kernel $E$. Since $E \subseteq \ker{\tilde \theta}$,
by \cite[Lemma (11.26)]{Isa}, we have that $(\tilde G,\tilde N,\tilde\theta)$
and $(G,N, \theta)$ are isomorphic character triples.
Moreover, this isomorphism  preserves $p$-rational characters.

Now, define $\lambda\in\irr{\tilde N}$ by $\lambda(n,\epsilon)=\epsilon^{-1}$,
so that $\lambda$ is a linear character of $\tilde N$ with $\ker{\lambda}=N$.
Observe that  $\lambda$ is $\tilde G$-invariant. Also,
$$\tau_{\tilde N}=\lambda^{-1}\tilde \theta \, .$$
By the remark after Lemma (11.27) of \cite{Isa}, we have that
$(\tilde G, \tilde N, \lambda)$ and $(\tilde G, \tilde N, \tilde \theta)$ are
isomorphic character triples. The bijection
$\irr{\tilde G|\lambda} \rightarrow \irr{\tilde G|\tilde \theta}$
is given by
$\beta \mapsto \beta\tau$. Furthermore, if $N \subseteq U \subseteq G$ and
$\chi\in\irr{U \mid\theta}$, then the character corresponding to $\chi$ under
the concatenation of the two character triple isomorphisms is a character
$\beta\in\Irr(\tilde U|\lambda)$ where $U=\tilde Z /E$ and $ \beta\tau_U$ is
a lift of $\chi$. (The corresponding facts about Brauer characters are
precisely quoted in the proof of \cite[Thm.~(8.28)]{N}.)

Now, if $\beta$ is $p$-rational, we have that $\beta\tau_U$ is $p$-rational.
If $\beta\tau_U$ is $p$-rational, and $\sigma\in\Gal(\QQ_{|G|}/\QQ_{|G|_{p'}})$,
then $\beta$ and $\beta^\sigma$ have the same image under the bijection
$\beta \mapsto \beta\tau_U$. Thus $\beta=\beta^\sigma$, and $\beta$ is
$p$-rational. Now, $(\tilde G, \tilde N, \lambda)$ and
$(\tilde G/N, \tilde N/N, \lambda)$ are isomorphic by \cite[Lemma (11.26)]{Isa}
with an isomorphism which preserves $p$-rationality. Now, write
$\Gamma=\tilde G/N$, $M=\tilde N/N$ and $\xi=\lambda$,
and observe that $M$ is a central $p'$-subgroup of $\Gamma$.
\end{proof}

\section{Proof of the main theorem}   \label{sec:main proof}

In this section we prove the Main Theorem, based on
Theorem~\ref{thm:quasisimples} for quasi-simple groups, shown in
Sections~\ref{sec:notLie} and~\ref{sec:Lie}. Moreover we apply results of
Section~\ref{prational} for the construction of a $p$-rational lift.

Let $G$ be a finite group and $p$ be a prime. We first derive some direct
consequences of Theorem~\ref{thm:quasisimples}.

\begin{cor}   \label{cor:semisimple}
 Suppose that $E$ is a perfect group such that $E/\bZ(E)$ is a direct product
 of simple groups. Let $A=\Aut(E)$. Let $p$ be a prime, and assume that
 $|\bZ(E)|$ is prime to $p$. Let $b$ be a $p$-block of $E$.
 If $\IBr(b)$ consists of $A$-conjugates of $\vhi$, then there exists a
 $\chi \in \Irr(E)$ such that $\chi^0=\vhi$, the stabilisers
 $I_A(\chi)=I_A(\vhi)=:I$ coincide, and $\chi$ extends to some
 $\tilde \chi\in\Irr(E\rtimes P)$ for some $P \in \Syl_p(I)$ such that
 $\ker(\tilde \chi)$ contains a Sylow $p$-subgroup of $\cent {E\rtimes P} E$.
 If $p\ne2$ then $\chi$ can be chosen to be $p$-rational in addition.
\end{cor}

\begin{proof}
Write $Z=\bZ(E)$. By the \emph{universal $p'$-cover} of a perfect group $S$ we
mean here the quotient of a universal covering group $\widehat S$ of $S$ by the 
Sylow $p$-subgroup of $\ZZ(\widehat S)$. By the theory of covering groups,
there are non-isomorphic simple groups $S_1,\ldots,S_t$ with universal
$p'$-covers $\hat{S_i}$, integers $n_i\ge1$ and an epimorphism
$$\pi:G:=\hat{S_1}^{n_1} \times \cdots \times \hat{S_t}^{n_t} \rightarrow E,$$
with $K=\ker\pi \le\bZ(G)$, and $\pi(\bZ(G))=\bZ(E)$. Via $\pi$ the group
$\Aut(E)$ can be identified with a subgroup of $\Aut(G/\bZ(G))=\Aut(G)$.
Hence, we may work in $G$, and assume that $G=E$.

Now $b$ is a $p$-block of $G$, and therefore there are blocks $b_{i,j}$
($1\leq j\leq n_i$) of $\hat{S_i}$, and natural bijections
$\Irr(b_{1,1}) \times \cdots \times \Irr(b_{t,n_t}) \rightarrow \Irr(b)$ and
$\IBr(b_{1,1}) \times \cdots \times \IBr(b_{t,n_t}) \rightarrow \IBr(b)$.
Without loss of generality we may assume that any two
$\Aut(\hat {S_i})$-conjugate blocks $b_{i,j}$ and $b_{i,j'}$ are equal. Now,
$\Aut(G)=\Aut(\hat{S_1})\wr\fS_{n_1}\times\cdots\times\Aut(\hat{S_t})\wr\fS_{n_t}$,
and our assumption forces that all elements of $\IBr(b_{i,j})$ are
$\Aut(\hat{S_i})_{b_{i,j}}$-conjugate. By Theorem~\ref{thm:quasisimples}
there exist $\vhi_{i,j}\in\IBr(b_{i,j})$ and $\chi_{i,j} \in \Irr(b_{i,j})$
such that $(\chi_{i,j})^0=\vhi_{i,j}$,
and $I_{\Aut(\hat{S_i})}(\chi_{i,j})=I_{\Aut(\hat{S_i})}(\vhi_{i,j})$.
Then
$$\chi:=\chi_{1,1} \times \cdots \times \chi_{t,n_t}\in\Irr(G)$$
is a lift of
$$\vhi_{1,1} \times \cdots \times \vhi_{t,n_t} \in \IBr(b) \, .$$
If $p\ne 2$ the characters $\chi_{i,j}$ and hence $\chi$ are $p$-rational and
the claim follows from \cite[Thm.~(6.30)]{Isa}.

The following considerations prove the statement for $p=2$.
Let $P_{i,j}$ be a Sylow $p$-subgroup of $I_{\Aut(\hat{S_i})}(\chi_{i,j})$. By
Theorem~\ref{thm:quasisimples}(4) the character $\chi_{i,j}$ extends to some
$\tilde \chi_{i,j}\in \Irr(\hat{S_i} \rtimes P_{i,j})$ such that
$p\nmid |\cent {\hat {S_i} \rtimes P_{i,j}} {\hat {S_i}}:
\ker(\tilde \chi_{i,j}|_{\cent {\hat {S_i} \rtimes P_{i,j}}{\hat {S_i}}})|$.

Without loss of generality we may assume that whenever two characters
$\vhi_{i,j}$ and $\vhi_{i,j'}$ seen as characters of $\hat {S_i}$ are
$\Aut(\hat {S_i})$-conjugate they are equal, and that the same applies to
the characters $\chi_{i,j}$ and $\chi_{i,j'}$. This implies that $\Aut(G)_\chi$
is of the form
$$\Aut(G)_\chi=\left (\Aut(\hat {S_1})_{\chi_{1,1}} \times \cdots
  \times\Aut(\hat {S_t})_{\chi_{t,n_t}}\right ) \rtimes\cS,$$
where $\cS$ is a direct product of symmetric group permuting some isomorphic
factors of $G$. A Sylow $p$-subgroup of $G\rtimes \Aut(G)_\chi$ is contained in
$$A:=\left ((\hat{S_1}\rtimes P_{1,1}) \times\cdots
  \times (\hat{S_t}\rtimes P_{t,n_t})\right  ) \rtimes\cS.$$
Following \cite[Thm.~25.6]{Hu} the character
$$\tilde\chi_{1,1} \times \cdots \times \tilde \chi_{t,n_t}
  \in\Irr((\hat{S_1}\rtimes P_{1,1}) \times\cdots\times
  (\hat{S_t}\rtimes P_{t,n_t}))$$
has an extension $\tilde\chi$ to $A$ with
$p\nmid|\cent {G\rtimes \Aut (G)_\chi} G:\ker(\tilde\chi|_{\cent {G\rtimes \Aut (G)_\chi} G})|$ by construction.
\end{proof}

\begin{cor}   \label{situation}
 Let $G$ be a finite group and $E\nor G$ be perfect with $E/\bZ(E)$
 a direct product of simple groups. Assume that $\cent GE$ has order prime to
 $p$. Let $b$ be a $p$-block of $E$, and assume that $\IBr(b)$ consists of
 $G$-conjugates of $\vhi$. Then there exists some $\chi\in\Irr(E)$ such that
 $\chi^0=\vhi$, the stabilisers $I_G(\chi)=I_G(\vhi)=I$ coincide, and $\chi$
 extends to $EP$ for $P \in \Syl_p(I)$.
 Moreover, if $p\ne 2$ the character $\chi$ can be chosen to be $p$-rational.
\end{cor}

\begin{proof}
We have that $G/\cent GE$ is isomorphic to a subgroup $C$ of $A=\Aut(E)$.
Now, $\IBr(b)$ consists of $A$-conjugates of $\vhi$. Therefore by
Corollary~\ref{cor:semisimple}, there exists $\chi \in \Irr(E)$ such that
$\chi^0=\vhi$, $I_A(\chi)=I_A(\vhi)$ and $\chi$ extends to some
$\tilde \chi\in\Irr(E\rtimes P)$ for $P \in \Syl_p(I)$ such that
$\ker(\tilde \chi)$ contains a Sylow $p$-subgroup of $\cent {E\rtimes P} E$.
Hence $\chi$ and $\vhi$ have the same stabiliser in $C$ and
$I_G(\chi)=I_G(\vhi)$. By Lemma \ref{extendingclean} we have that $\chi$
extends to $EP$. If $p\neq 2$ then $\chi$ can be chosen to be $p$-rational
by Corollary~\ref{cor:semisimple}.
\end{proof}


\begin{lem}   \label{lem:solv}
 Suppose that $E\le L$ are normal in $G$, with $L/E$ solvable. Suppose that
 $\chi\in\Irr(E)$ is $G$-invariant, and that $\chi^0=\vhi \in \IBr(E)$.
 Assume that $\chi$ extends to $P$, where $P/E$ is a Sylow $p$-subgroup
 of $G/E$. Suppose that $\eta\in\IBr(L|\vhi)$ is $G$-invariant.
 Then there exists $\delta \in \Irr(L|\chi)$ such that $\delta^0=\eta$,
 $\delta$ is $G$-invariant, and $\delta$ extends to $Q$,
 where $Q/L\in\Syl_p(G/L)$. Furthermore, if $p$ is odd and $\chi$ is
 $p$-rational, then $\delta$ can be chosen to be $p$-rational, too.
\end{lem}

\begin{proof}
By Problem (8.12) of \cite{N}, we may assume that
$E\le\bZ(G)$. Using that $\chi$ extends to a Sylow $p$-subgroup of $G$, we may
assume that $E$ is a $p'$-group. (Use the ideas of Problem (8.13) of \cite{N}.)
We have now that $L$ is $p$-solvable. In \cite{I}, Isaacs constructed a
canonical subset $B_{p'}(L) \subseteq \Irr(L)$ with certain remarkable
properties. Among others, restriction to $p$-regular elements of $L$
provides a canonical bijection $B_{p'}(L) \rightarrow {\rm IBr}(L)$
(by \cite[Cor.~(10.3)]{I}). Also, by definition, characters in $B_{p'}(L)$
are $p$-rational. Now,  let $\delta \in \Irr(L)$ be the character in
$B_{p'}(L)$ with $\delta^0=\eta$. Since $E$ is a $p'$-group, we have
that $\delta$ lies over $\chi$. Since $\eta$ is $G$-invariant, then $\delta$
is $G$-invariant (again using \cite[Cor.~(10.3)]{I}). Since $\delta\in B_{p'}(L)$,
then $\delta$ extends to $Q$ by \cite[Cor.~(6.3)]{Isa}. If $p$ is odd and $\chi$
is $p$-rational, then we use the same argument as before, but applying
Theorem~\ref{thm:prationalisom}.
\end{proof}

We are now ready to prove the main theorem, which we restate:

\begin{thm}
 Let $G$ be a finite group and let $p$ be a prime. Assume that $B$ is a
 $p$-block of $G$ with $\IBr(B)=\{\vhi\}$. Then there exists a character
 $\chi\in\Irr(G)$ with $\chi^0=\vhi$ and $\chi$ is $p$-rational if $p\ne 2$.
\end{thm}

\begin{proof}
We argue by induction on $|G:\bZ(G)|$.
We may assume that $\bO_p(G)=1$ since $\bO_p(G) \subseteq \ker(\vhi)$.
By the Fong--Reynolds Theorem \cite[Thm.~(9.14)]{N}), we may assume that
$B$ is quasi-primitive, i.e. every block of every normal subgroup
of $G$ covered by $B$ is $G$-invariant.

In particular, when $N\nor G$ and $B$ covers $\{\theta\}$, where
$\theta\in\Irr(N)$ has defect zero, then $\theta$ is $G$-invariant. According to
Theorem ~\ref{thm:prationalisom} we may
assume that $N$ is a central $p'$-subgroup of $G$ after using
ordinary-modular $p$-rational triples. Thus, if $S$ is the largest
solvable normal subgroup of $G$, we have that the Fitting subgroup $\bF(S)$
is central, and therefore $S=Z:=\bZ(G)$.

Let $b$ be the block of $E=\bE(G)$, the group of components of $G$, covered by
$B$. Then $\IBr(b)$ consists of $G$-conjugates of some $\theta \in \IBr(b)$.
Let $Z_0:=\bZ(E)=Z \cap E$.

Write $E/Z_0=E_1/Z_0 \times \cdots \times E_t/Z_0$, where $E_i/Z_0$ is the
direct product of $G$-conjugates of distinct simple groups $S_i/Z_0$ and
$E_i \nor G$. Recall that $[E_i,E_j]=1$ if $i\ne j$. Let $D$ be a defect group
of $b$. We claim that $E\cent GD/E$ is solvable.
Write $DZ_0/Z_0=D_1Z_0/Z_0 \times \cdots \times D_tZ_0/Z_0$, where
$D_i=D\cap E_i$. We have that $D_i$ is a defect group of the $G$-invariant
block $b_i$ of $E_i \nor G$ covered by $b$. By our assumption, we have that
$D_i>1$ for every $i$, since $B$ does not cover defect zero blocks of
non-central normal subgroups. Also, $G=E_i\norm G{D_i}$ by the Frattini
argument. Let $U/Z_0$ be a simple factor of $E_i/Z_0$, so that
$$E_i/Z_0=(U/Z_0)^{x_1} \times \cdots \times (U/Z_0)^{x_t} \, ,$$
for some $x_i \in \norm G{D_i}$. Now, it is well-known that
$$D_i=(D_i \cap U)^{x_1} \cdots (D_i \cap U)^{x_r} \,  $$
and we see that $D_i \cap U>1$. Since $x \in C:=\cent GD\le \cent G{D_i}$
centralises $D_i\cap U>1$, it
follows that $x$ normalises $U$ (since $G$ transitively permutes the
simple factors of $E_i/Z_0$). We know that $\norm GU/U\cent GU$ is solvable.
Hence $U\cent GU \cap C \nor C$  and its quotient is solvable.
Since this is true for all components $U$ of $G$, we have that
$$C/(C \cap \bigcap_U U\cent GU)$$
is solvable, where $U$ runs over the components of $G$. Notice that
$$\bigcap_U U\cent GU={\bF^*}(G) \, ,$$
as every component commutes with every other component and with
every normal solvable subgroup.

Now, let $M=EC=\bF^*(G)C \nor G$, where $M/E$ is solvable. (Recall that
$G=E\norm GD$ by the Frattini argument.) Let $e$ be the unique block of $M$
covered by $B$. Let $\eta \in \IBr(M)$ under $\vhi$, and let
$\theta \in \IBr(b)$ under $\eta$. Necessarily $\eta \in \IBr(e)$. Now, if
$Q$ is a defect group of $e$, then we may assume that $Q$ contains $D$, and
therefore $\cent GQ \le \cent GD\le M$. Therefore, $B$ is the only block of $G$
covering $e$ by \cite[Lemma~3.1]{NT}.

Now, let $T$ be the stabiliser of $\eta$ in $G$, and let $I$ be the stabiliser
of $\theta$ in $G$. By the Frattini argument, we have that $T=M(I\cap T)$.
Let $\nu\in\IBr(T|\eta)$ be the Clifford correspondent of $\vhi$ over $\eta$.
Now, if $\delta \in \IBr(T|\eta)$, we have that $\delta^G \in \IBr(G)$.
Notice that $\delta^G$ lies in a block that covers $e$ by \cite[Thm.~(9.2)]{N},
but this must be $B$. Hence $\delta^G=\vhi$ by hypothesis, and we have that
$\delta=\nu$, by the uniqueness in the Clifford correspondence. Hence, $\eta$
is fully ramified in $T/M$. Now, let $\nu' \in \IBr(I\cap T|\theta)$ and
$\eta' \in \IBr(I\cap M|\theta)$ be the Clifford correspondents
of $\nu$ and $\eta$ over $\theta$, respectively. Then,
$$(\nu(1)/\eta(1)) \eta=\nu_M=((\nu')^T)_M=((\nu')_{I\cap M})^M \, ,$$
and we deduce that $\nu'$ is fully ramified with respect to $\eta'$.

By hypothesis, all elements in $\IBr(b)$ are $G$-conjugate. Thus, by
Corollary~\ref{situation}, there are $\theta_1\in\IBr(b)$ and $\chi\in\Irr(b)$
such that $\chi^0=\theta_1$, with same stabiliser $I_1$ in $G$, and such
that $\chi$ extends to a Sylow $p$-subgroup $P/E$ of $I_1/E$. If $p\ne2$,
$\chi$ is in addition $p$-rational. Since all elements of $\IBr(b)$ are
$G$-conjugate, we may assume that $\chi^0=\theta$, that $\chi$ is
$I$-invariant, and that $\chi$ extends to $P$, where $P/E \in \Syl_p(I/E)$.

Now, by Lemma~\ref{lem:solv}, there exists $\eta_1 \in \Irr(I\cap M)$
such that $(\eta_1)^0=\eta'$, $\eta_1$ is $I\cap T$-invariant, and $\eta_1$
extends to $Q$, where $Q/(I\cap M) \in \Syl_p((I\cap T)/(I\cap M))$.
Now, as in the proof of Lemma \ref{lem:solv}, there exists a group $X$ with
a central $p'$-subgroup $Y$, with a linear character $\lambda \in \Irr(Y)$
such that the triples $(I\cap T, I\cap M, \eta_1)$ and $(X,Y, \lambda)$ are
isomorphic as ordinary-modular triples.
If $p\ne 2$ the isomorphism of character triples has the properties
from Theorem \ref{thm:prationalisom}. Now, notice that if $\lambda^X=f\delta$,
with $f\ge1$ and $\delta\in\IBr(X)$, and $B'$ is the block of $\delta$, then
$\IBr(B')=\{\delta\}$ (see, for instance, the proof of \cite[Thm.~2.1]{NST}).
Since $|X:\bZ(X)|<|G:\bZ(G)|$, we have that $\delta$ is liftable by induction
and liftable to a $p$-rational character in case of $p\ne2$. Hence, we have
that $\nu'$ is liftable, even to  a $p$-rational character for $p\ne2$.
Thus, there exists $\rho \in \Irr(I\cap T)$ such that
$\rho^0=\nu'$. Now, $(\nu')^T=\nu$ and thus $\nu^G=\vhi$, so $\nu'^G=\vhi$,
and $(\rho^G)^0=\nu'^G=\vhi$. We therefore deduce that $\rho^G$ is a lift
of $\vhi$ and $p$-rational if $p\ne2$.
\end{proof}


\end{document}